\documentclass[11pt]{article}
\usepackage{epsfig}
\usepackage{graphicx}
\usepackage{color}
\usepackage{amsmath,amsthm,amssymb}
\usepackage{latexsym}
\newtheorem{theor}{Theorem}[section]
\newtheorem{theo}[theor]{Theorem}

\newtheorem{lemma}[theor]{Lemma}
\newtheorem{fact}[theor]{Fact}
\newtheorem{defi}{Definition}[section]
\usepackage{hhline}

  \baselineskip 8 mm

\begin{document}

\title{Boxicity of Circular Arc Graphs}
\author {Diptendu Bhowmick
\thanks {Computer Science and Automation Department, Indian
Institute of Science, Bangalore- 560012 Email:
diptendubhowmick@gmail.com},
L. Sunil Chandran
\thanks{Computer Science and Automation Department, Indian Institute
of Science, Bangalore- 560012 Email: sunil@csa.iisc.ernet.in} }
\date{}
\maketitle

\begin{abstract}
A $k$-dimensional box is the cartesian product $R_1 \times R_2
\times \cdots \times R_k$ where each $R_i$ is a closed interval on
the real line. The {\it boxicity} of a graph $G$, denoted as
$box(G)$, is the minimum integer $k$ such that $G$ can be represented as the
intersection graph of a collection of $k$-dimensional boxes: that is
two vertices are adjacent if and only if their corresponding boxes
intersect. A circular arc graph is a graph that can be represented
as the intersection graph of arcs on a circle.\\
  Let $G$ be a circular arc graph with maximum degree $\Delta$.
We show that if $\Delta <\left \lfloor \frac{n(\alpha-1)}{2\alpha}\right \rfloor$, $\alpha \in \mathbb{N}$, $\alpha
\geq 2$ then $box(G) \leq \alpha$. We also demonstrate a graph with boxicity $> \alpha$ but with $\Delta=n\frac{(\alpha-1)}{2\alpha}+\frac{n}{2\alpha(\alpha+1)}+(\alpha+2)$. So the result cannot be improved substantially when $\alpha$ is large.\\
Let $r_{inf}$ be minimum number of arcs passing through any point on the circle with respect to some circular arc representation of $G$. We also show that for any circular arc
graph $G$, $box(G) \leq r_{inf} + 1$ and this bound is tight.\\
Given a family of arcs $F$ on the circle, the circular cover number $L(F)$ is the cardinality of the smallest subset $F'$ of $F$ such that the arcs in $F'$ can cover the circle. Maximum circular cover number $L_{max}(G)$ is defined as the maximum value of $L(F)$ obtained over all possible family of arcs $F$ that can represent $G$. We will show that if $G$ is a circular arc graph with $L_{max}(G)> 4$ then $box(G) \leq
3$.\\

\end{abstract}

\section{Introduction}

Let $\mathcal{F}$ be a family of non-empty sets. An undirected graph
$G$ is the intersection graph of $\mathcal{F}$ if there exists a one-one correspondence between the vertices of $G$ and the sets in $\mathcal{F}$ such that two vertices in $G$ are adjacent if and only
if the corresponding sets have non-empty intersection. If
$\mathcal{F}$ is a family of intervals on the real line, then $G$ is
called an {\it interval graph}.

A $k$-dimensional box or $k$-box is the cartesian product $R_1
\times R_2 \times \cdots \times R_k$, where each $R_i$ is a closed
interval on the real line. The boxicity of a graph $G$ is defined to
be the minimum integer $k$ such that $G$ is the intersection graph
of a collection of $k$-boxes. Since $1$-boxes are nothing but closed
intervals on the real line, interval graphs are the graphs having
boxicity $1$.

The concept of boxicity was introduced by F. S. Roberts
\cite{Roberts} in 1969. Boxicity finds applications in fields such
as ecology and operations research: It is used as a measure of the complexity of ecological $\cite{Roberts1}$ and social $\cite{freeman}$ networks and has applications in fleet maintenance $\cite{opsut}$. Boxicity has been investigated for various classes of graphs $\cite{Scheiner}\cite{Roberts2}\cite{Thomassen}$ and has been related with other parameters such as treewidth $\cite{treewidth}$ and vertex cover $\cite{vertexcover}$. Computing the boxicity of a
graph was shown to be NP-hard by Cozzens \cite{Cozzens}. This was
later strengthened by Yannakakis \cite{Yan1}, and finally by
Kratochvil \cite{Kratochvil} who showed that deciding whether
boxicity of a graph is at most two itself is NP-complete. Recently Chandran et al $\cite{maxdegree}$ showed that for any graph $G$, $box(G) \leq \chi(G^2)$ where $G^2$ is the square of graph $G$ and $\chi$ is the chromatic number of a graph. From this they inferred that $box(G) \leq 2\Delta^2+2$, where $\Delta$ is the maximum degree of $G$. Very recently this result was improved by Esperet $\cite{esperet}$, who showed that $box(G) \leq \Delta^2+2$. In $\cite{geometric}$ Chandran et al have shown that $box(G) \leq \left \lceil(\Delta+2)\log n \right \rceil$ where n is the number of vertices in $G$.\\

\noindent{}\textbf{Remark on previous approach to boxicity of circular arc graphs:}\\
In $\cite{treewidth}$ it has been shown that for any graph $G$, $box(G) \leq treewidth(G)+2$. If $G$ is circular arc graph then $treewidth(G) \leq 2\omega(G) -1$. It follows that $box(G) \leq 2\omega(G)+1$. But the results shown in this paper are much stronger.      

\subsection{Our results}

A graph $G$ is called a {\it circular arc} graph if it is the
intersection graph of a collection of arcs on the circle that is, 
each vertex of this graph can be associated to an arc on a circle
such that two vertices are adjacent if and only if the arcs
corresponding to these vertices intersect. Circular
arc graphs were first discussed as a natural generalization of
interval graphs and they have since been studied extensively. See
Golumbic\cite{Golu} for a brief introduction and references on
circular arc graphs. It is easy to see that a {\it circular arc}
representation of an undirected graph $G$ which fails to cover some
point $P$ on the circle will be topologically the same as an
interval representation of $G$.

Since the definitions of interval graphs and circular arc graphs are
quite similar, one may tend to conjecture that the boxicity of
circular arc graphs may not be very high. But it comes as a surprise
that the graph that achieves the highest boxicity belongs to the
class of circular arc graphs! It was shown by Roberts
$\cite{Roberts}$ in his pioneering paper that the highest possible
value of boxicity namely $\frac{n}{2}$ (Let n be an even number) was achieved by
$\overline{(\frac{n}{2})K_2}$, the complement of a perfect matching on
n vertices. (Note that this graph is same as the complete $\frac{n}{2}$ partite graph on n vertices with each partition containing exactly two vertices). It turns out that $\overline{(\frac{n}{2})K_2}$ is a circular arc graph. A circular arc representation of $\overline{(\frac{n}{2})K_2}$ is given below. In this paper we will show that boxicity of a circular arc graph can be large only when its maximum degree is very high.\\

\noindent{\bf A circular arc representation of $\overline{(\frac{n}{2})K_2}$ :}\\
Let us take a circle. Place n points on the circle such that the distance between any adjacent pair of points is equal. Let these points be $p_0,p_1, \ldots, p_{n-1}$ in clockwise direction. Draw an arc from $p_i$ to $p_{(i+\frac{n}{2}-1) \ mod \ n}$ in clockwise direction for $0 \leq i <n$. It is easy to verify that the arc starting from $p_i$ is adjacent to all other arcs except the arc starting from $p_{(i + \frac{n}{2}) \ mod \ n }$. Thus this family of arcs gives a circular arc representation for $\overline{(\frac{n}{2})K_2}$. Note that the vertices corresponding to the arcs starting from $p_i$ and $p_{(i + \frac{n}{2}) \ mod \ n}$ belong to the same partition in $\overline{(\frac{n}{2})K_2}$ for $0 \leq i < \frac{n}{2}$, when we visualize $\overline{(\frac{n}{2})K_2}$ as a complete $\frac{n}{2}$partite graph.\\\\
Let $G$ be a circular arc graph on n vertices and with maximum
degree $\Delta$. In this paper we will show the following results:\\

\noindent{\bf Result 1.} If $\Delta$ $< \left \lfloor n(\frac{\alpha-1}{2\alpha}) \right \rfloor$ then
$box(G)$ $\leq$ $\alpha$, $\alpha \in \mathbb{N},$ $\ \alpha \geq 2$.

\noindent{\bf Remark:} We are unable to show that the range of
$\Delta$ as given by the above result, for which $box(G) \leq \alpha$, is tight. But we can construct
a circular arc graph with $\Delta=n\frac{(\alpha-1)}{2\alpha}+\frac{n}{2\alpha(\alpha+1)}+(\alpha+2)$ and
$box(G) > \alpha$. Thus the result cannot be substantially improved when $\alpha$ is large.\\

\noindent{\bf Result 2.} Given a family of arcs $F$ on the circle, $r_{inf}(F)$ is the minimum number of arcs passing through any point on the circle. Given a circular arc graph $G$, $r_{inf}(G)=min \ r_{inf}(F)$ over all possible family of arcs $F$ that can represent $G$. We will show that $box(G)$ $\leq r_{inf}(G) + 1$ and this bound is tight.\\

\noindent{\bf Result 3.} Given a family of arcs $F$ on the circle, the circular cover number $L(F)$ is the cardinality of the smallest subset $F'$ of $F$ such that the arcs in $F'$ can cover the circle. Maximum circular cover number $L_{max}(G)$ is defined as the maximum value of $L(F)$ obtained over all possible family of arcs $F$ that can represent $G$. We will show that if $G$ is a circular arc graph with $L_{max}(G)> 4$ then $box(G) \leq
3$.

\section{Preliminaries}

Let $G$ be a simple, finite, undirected graph on $n$ vertices. The
vertex set of $G$ is denoted as $V(G)$ and the edge set of $G$ is
denoted as $E(G)$. Let $G'$ be a graph such that $V(G') = V(G)$.
Then, $G'$ is a {\it super graph} of $G$ if $E(G) \subseteq E(G')$.
We define the {\it intersection} of two graphs as follows: if $G_1$
and $G_2$ are two graphs such that $V(G_1) = V(G_2)$, then the
intersection of $G_1$ and $G_2$ denoted as $G = G_1 \cap G_2$ is a
graph with $V(G) = V(G_1) = V(G_2)$ and $E(G) = E(G_1) \cap E(G_2)$.

Let $G$ be a graph. Let $I_1, I_2, \ldots, I_k$ be $k$ interval
graphs such that $G = I_1 \cap I_2 \cap \cdots \cap I_k$. Then $I_1,
I_2, \ldots, I_k$ is called an {\it interval graph representation}
of $G$. The following equivalence is well known.

\begin{lemma}\textbf{$($Roberts\cite{Roberts}$)$} 
The minimum $k$ such that there exists an interval graph
representation of $G$ using $k$ interval graphs $I_1, I_2, \ldots,
I_k$ is the same as $box(G)$.
\end{lemma}

\begin{defi}\label{Overlap set}
Given a family F of arcs on a circle, the $\textbf{overlap set}$ of a point P on the circle is the set of all arcs that contains the given
point P and is denoted by $\mathcal{O}(P)$. An overlap set
with the smallest $($respectively largest$)$ number of arcs in it is called a $\textbf{minimum}$ $($respectively \textbf{maximum}$)$
$\textbf{overlap set}$ and its cardinality is denoted by $r_{inf}(F)$ $($respectively $r_{sup}(F))$.
\end{defi}

It is easy to see that the arcs in an overlap set induce a clique in
the corresponding circular arc graph.

\begin{defi}
 For a circular arc graph $G$,  $r_{inf}(G)$ is defined as the minimum value of $r_{inf}(F)$ obtained over all possible family of arcs $F$ that can represent $G$.  
\end{defi}

\begin{defi}\label{Minimum overlap point}
A point $P$ is called a $\textbf{minimum overlap point}$ if
$|\mathcal{O}(P)|=r_{inf}(F)$. Note that there may be many minimum
overlap points. Let $P_{inf}$ denote one of the minimum overlap
points. All through the paper we shall use this special minimum
overlap point as a reference point.
\end{defi}

\begin{defi}
Each arc has two endpoints. The $\textbf{left endpoint} \ $
$l(u)$ $($respectively $\textbf{right endpoint}$ $r(u))$ of arc u is the
first endpoint of u encountered in an anticlockwise $($respectively
clockwise$)$ traversal from any interior point of u. $($The circle
itself is not considered as an arc. A single point is also not
considered as an arc. Thus the definition makes sense and every arc
has a distinct left as well as a right end point. Note that these
assumptions can be made without any loss of generality.$)$
\end{defi}

\begin{defi}\label{Clockwise and anticlockwise adjacency}
An arc v is said to be $\textbf{clockwise adjacent}$ to an arc u if
v belongs to the overlap set of $\ l(u)$. An arc v is said to be
$\textbf{anticlockwise adjacent}$ to an arc u if v belongs to the
overlap set of $\ r(u)$.
\end{defi}

\begin{defi}
The $\textbf{start angle}$ $s(u) \ ($respectively $\textbf{end angle}$
$e(u))$ of arc u is the angle measured in clockwise direction from the
line joining $l(u) \ ($respectively $r(u))$ and center of the circle to the
positive x-axis assuming that the center is at the origin.
\end{defi}

\begin{defi}\label{clockwise ordering}
Consider the set of end points $\{l(u), r(u): u \in V \}$. A
$\textbf{clockwise}$ $\textbf{ordering}$ $\sigma$ of the endpoints
is obtained, if we traverse the circle in the clockwise direction
starting from the point $P_{inf}$ and list out the endpoints in the
order in which they are seen, resolving any conflict arbitrarily. For a circular arc graph $G$ with a given circular arc representation, a clockwise ordering of the vertices can be obtained from the ordering of the left endpoints of their corresponding arcs in a similar way. In this paper we shall use $v_i$ to denote the $i$th vertex with respect to the clockwise ordering.
\end{defi}
\begin{defi}
Let $Int(\textbf{u})$ denote the set of half axes intersecting the arc \textbf{u}.
Then $|Int(\textbf{u})|$ is said to be the $\mathbf{interception \ number}$
of arc \textbf{u}.
\end{defi}
Let $G$(V,E) be a circular arc graph with circular arc representation
$\mathcal{C}_0$. From now on, the arc corresponding to vertex u will be denoted as $\textbf{u}$. We assume $n \geq 2$ and $\Delta \geq 1$. Let us take a circle of unit radius centered at origin on the cartesian plane.
Given a positive integer $\alpha \geq 2$, we
define $\alpha$ reference axes as follows. The first reference axis $A_0$ is the same
as the x-axis. For $1 \leq j < \alpha$ the $j$th reference axis $A_j$ is the
line passing through the origin and making an angle $\frac{\pi j}{\alpha}$ radian with the positive x-axis in anticlockwise direction. For the convenience of writing, we may sometimes use the notation $A_j$ even when $j \geq \alpha$: In that case $A_j$ denotes the axis $A_{j \ \mod \ \alpha}$. Also sometimes we use the notation $A_{-i}$ to denote $A_{\alpha-i}$. (We will refer to this set of reference axes as the system of reference axes based on $\alpha$).\\\\
Origin divides each reference axis into two halves. Each of them is called a half-axis. Clearly these $2\alpha$ half-axes
partition the interior of the circle into 2$\alpha$ regions - each region is
called a sector. Note that each sector defines an arc on the circle.
In this paper we will use the word sector to denote the arc defined
by a sector also. An $\textbf{anticlockwise}$ $\textbf{ordering}$ of the half-axes
is obtained, if we traverse the circle in anticlockwise direction starting from the positive x-axis. We shall use the notation $H_j$ to
denote the $j$th half-axis in this ordering where $0 \leq j <2\alpha$. It may be noted that the half axes obtained from the axis $A_j$ will get numbered $H_j$ and $H_{j+\alpha}$ according to the above mentioned numbering scheme. Sometimes we will refer $H_j$ as the positive half axis of $A_j$.\\
 The sector in-between the half-axes $H_j$ and $H_{(j+1) \ mod \ 2\alpha}$ is referred to as the $j$th sector and is denoted as $S_j$ for $0 \leq j < 2\alpha$. (We sometimes use the notation $S_{-i}$ to denote $S_{2\alpha-i}$). $H_{j}$ and $H_{(j+1) \ mod \ 2\alpha}$ are said to be the $\textbf{right}$ and $\textbf{left half-axis}$ of $S_j$ respectively. We also define $A_{j \ mod \ \alpha}$ and $A_{(j+1) \ mod \ \alpha}$ respectively to be the $\textbf{right}$ and $\textbf{left}$ $\textbf{reference axis}$ of $S_j$. Note that $A_{j \ mod \ \alpha}$ corresponds to the half-axis $H_j$ and $A_{(j+1) \ mod \ \alpha}$ corresponds to the half-axis $H_{(j+1) \ mod \ 2\alpha}$.
\begin{lemma}$\label{distinct endpoints}$
There exists a circular arc representation $\mathcal{C}$ of $G$ which satisfies the following properties\\
\vspace{-0.6cm}
\begin{enumerate}
 \item All endpoints of arcs are distinct.
 \item Left endpoints of the arcs are uniformly placed over the perimeter of the circle i.e. $|s(\textbf{u}_{i+1})-s(\textbf{u}_i)|=$ $\frac{2\Pi}{n}$ radian where $0 \leq i <n-1$ $($Recall that $s(\textbf{u})$ is the start angle of $\textbf{u})$.
 \item Endpoints of the arcs do not lie on any reference axis.
\end{enumerate}
\end{lemma}
\begin{proof}
We leave it to the reader to verify that given the circular arc representation $\mathcal{C}_0$ for $G$, there exists a circular arc representation $\mathcal{C}_1$ for $G$ such that all the endpoints of the arcs are at distinct points on the circle.\\
Let $\sigma$ be clockwise ordering of the endpoints of arcs in $\mathcal{C}_1$. Now given the circular arc representation $\mathcal{C}_1$ for $G$ and clockwise ordering $\sigma$
of the endpoints of the arcs, we shall show
that there is another circular arc representation $\mathcal{C}_2$ for G such
that $|s(\textbf{u}_{i+1})-s(\textbf{u}_i)|=$ $\frac{2\Pi}{n}$ radian for $0 \leq i < n-1$. Let $l_0,l_1, \ldots ,l_{n-1} $ be the
left endpoints of the arcs in the clockwise ordering $\sigma$. Place $l_i$ at the
point on the circle defined by the polar coordinates
$(1,-\frac{2\Pi}{n}i)$ for $0 \leq i <n$. Let $e_{i_1},e_{i_2}, \ldots ,e_{i_t}$ be
the right endpoints that are in-between $l_i$ and $l_{(i+1) \ mod \ n}$ with
respect to $\sigma$. Place $e_{i_1},e_{i_2}, \ldots e_{i_t}$ at
distinct points on the circle in that order in clockwise direction
on the arc in-between $l_i$ and $l_{(i+1) \ mod \ n}$. It is easy to verify
that $\mathcal{C}_2$ will be a valid circular arc representation of $G$.\\
Now with respect to $\mathcal{C}_2$ if some endpoint lies on some reference axis we can rotate the circular arc system on the circle by $\epsilon$ radian where $\epsilon$ is a sufficiently small positive real number, to get a circular arc representation $\mathcal{C}$ satisfying the desired properties.
\end{proof}
From now on we will use this circular arc representation $\mathcal{C}$ which satisfies the properties specified in Lemma $\ref{distinct endpoints}$. Let $\sigma$ be the clockwise ordering of the endpoints of arcs in $\mathcal{C}$.
\section{Boxicity when $\Delta < \left \lfloor n\frac{(\alpha-1)}{2\alpha} \right \rfloor$, $\alpha \in \mathbb{N}$, $\alpha \geq 2$}
In this section we will be considering circular arc graphs of maximum degree strictly less than $\left \lfloor n\frac{(\alpha-1)}{2\alpha} \right \rfloor$, $\alpha \in \mathbb{N}$, $\alpha \geq 2$. We will assume that we have a circular arc family $F$ that represents $G$ and obeys Lemma $\ref{distinct endpoints}$.\\\\
\noindent{}\textbf{Observation 1:} If $\Delta < \left \lfloor n\frac{(\alpha-1)}{2\alpha} \right \rfloor$, then each arc $u \in F$ is such that
\vspace{-0.4cm}
\begin{enumerate}
 \item length(u) $< \Pi(\frac{\alpha-1}{\alpha})$.
 \item $|Int(u)| \leq (\alpha-1)$. 
\end{enumerate}

\noindent{}\textbf{Proof of 1:} If an arc has length at least $\Pi(\frac{\alpha-1}{\alpha})$, then it will intersect the left endpoints of at least $\left \lfloor \frac{n(\alpha-1)}{2\alpha}\right \rfloor$ other arcs in view of Lemma $\ref{distinct endpoints}$. (Recall that we are considering a unit circle and therefore the distance between any pair of adjacent left endpoints is $\frac{2\Pi}{n}$).\\\\
\noindent{}\textbf{Proof of 2:} Since length(u) $< \Pi(\frac{\alpha-1}{\alpha})$, u can entirely contain at most $(\alpha-2)$ sectors. Therefore $|Int(u)| \leq (\alpha-1)$.\\  

All the definitions and Lemmas given in this section are applicable only for arcs of length $<\Pi(\frac{\alpha-1}{\alpha})$.
\begin{defi}
For every arc \textbf{u} the $\mathbf{extreme \ sectors}$ are the sectors containing the endpoints of arc \textbf{u}.
One of the extreme sectors will be called the $\mathbf{head \ sector}$ and the other will be called
the $\mathbf{tail \ sector}$ denoted by $head(\textbf{u})$ and $tail(\textbf{u})$ respectively. The sector $head(\textbf{u})$ is defined to be the extreme sector that contains $l(\textbf{u})$ and $tail(\textbf{u})$ is defined to be the extreme sector that contains $r(\textbf{u})$.
\end{defi}
With respect to a reference axis $A_j$
we define the image of a point $P$ on the circle as the point $P'$ on the circle such that
the line defined by $PP'$ is perpendicular to $A_j$. We denote $P'$
as $Im_j(P)$. When $P$ is on $A_j$ we take $Im_j(P)=P$. If X
is a set of points on the circle, we will use $Im_j(X)$ to denote
the set $\{ Im_j(P) : P \in X \}$.
\begin{lemma}
 Let X and Y be the $t$-th sector starting from $H_j$ $($i.e. the positive half axis of $A_j)$ in the clockwise and anticlockwise direction respectively. Then $Im_j(X)=Y$ and $Im_j(Y)=X$.
\end{lemma}
\begin{proof}
 Left to the reader.
\end{proof}
Note that the $t$th sector starting from $H_j$ for $0 \leq j < \alpha$ in the clockwise direction is $S_{(j-t) \ mod \ 2\alpha}$ whereas the $t$th sector starting from $H_j$ in the anticlockwise direction is $S_{(j+t-1) \ mod \ 2\alpha}$.
\begin{lemma}\label{Image}
 $Im_j(S_k)=S_h$ where $h=(2j-k-1) \ mod \ 2 \alpha$ for $0 \leq k <2 \alpha$ and $0 \leq j < \alpha$.
\end{lemma}
\begin{proof}
If $k<j$, $S_k$ is the $(j-k)$th sector from $H_j$ in clockwise direction and therefore $Im_j(S_k)$ will be the $(j-k)$th sector from $H_j$ in the anticlockwise direction namely $S_{2j-k-1}$. If $k \geq j$, $S_k$ is the $(k-j+1)$th sector from $H_j$ in anticlockwise direction. So $Im_j(S_k)$= $S_{(2j-k-1) \ mod \ 2 \alpha}$ as required.
\end{proof}
\begin{lemma}\label{Image2}
For any reference axis $A_j$, for $0 \leq j < \alpha$, $Im_j(l(\textbf{u}))=r(Im_j(\textbf{u}))$ and
$Im_j(r(\textbf{u}))=l(Im_j(\textbf{u}))$.
\end{lemma}
\begin{proof}
Left to the reader.
\end{proof}
\begin{defi}
With respect to a reference axis $A_j$
we define the $\textbf{projection}$ of a point $P$ on the circle as the point $P'$ on $A_j$ such that
the line defined by $PP'$ is perpendicular to $A_j$. We denote $P'$
as $Proj_j(P)$. When $P$ is on $A_j$ we take $Proj_j(P)=P$. If X
is a set of points on the circle, we will use $Proj_j(X)$ to denote
the set $\{ Proj_j(P) : P \in X \}$.
\end{defi}
Note that the reference axis $A_j$ can be considered to be obtained by rotating the x-axis (i.e. the real line) about the origin by an angle of $\frac{\pi j}{\alpha}$ radian in the anticlockwise direction. Then clearly there is a natural bijection between the points of $A_j$ and the set of real numbers namely the one which maps the positive x-axis to the positive half axis $H_j$ of $A_j$. For X $\subset A_j$ let $X' \subset \mathbb{R}$ be the set of real numbers that corresponds to X with respect to this bijection. Then we define $\inf{X}$ (respectively $\sup{X}$) to be the point on $A_j$ that corresponds to $\inf{X'}$ (respectively $\sup{X'}$).
\begin{fact}\label{projpoint}
If P and $P'$ are points on the circle such that $Proj_j(P)=Proj_j(P')$ for some j, where $0 \leq j < \alpha$. Then either $Im_j(P)=P'$ or $P=P'$.
\end{fact}
\begin{proof}
Left to the reader.
\end{proof}
\begin{lemma}\label{proj}
 Let P, $P'$ be points on the circle such that $P \neq P'$, $P \in S_k$ $($where $0 \leq k <2 \alpha)$ and $Proj_j(P)=Proj_j(P')$ $($where $0 \leq j < \alpha)$. Then $P' \in S_h$ where $h=(2j-k-1) \ mod \ 2 \alpha$.
\end{lemma}
\begin{proof}
By fact $\ref{projpoint}$, $S_h=Im_j(S_k)$. Then $h=(2j-k-1) \ mod \ 2 \alpha$ by Lemma $\ref{Image}$.
\end{proof}
\begin{defi}
If $Int(\textbf{u}) \neq \emptyset$ then let $H_{j_1}, H_{j_2},\ldots,H_{j_t}$ be the set of half axes in
$Int(\textbf{u})$ in the order in which they appear, as one traverses the arc in the anticlockwise direction. $($It may be noted that if the arc is properly contained in one sector then $Int(\textbf{u})$=$\emptyset)$. Now $\textbf{median half axis}$ $($respectively \textbf{median reference axis}$)$ with respect to arc \textbf{u} is defined to be $H_{j_{\left \lceil \frac{|Int(\textbf{u})|}{2} \right \rceil}}$ $($respectively $A_{j_{\left \lceil \frac{|Int(\textbf{u})|}{2} \right \rceil} \ \mod \ \alpha})$. We will use the notation $H_{m(\textbf{u})}$ and $A_{m(\textbf{u}) \ mod \ \alpha}$ respectively to denote the median half axis and median reference axis with respect to arc \textbf{u}.  
\end{defi}
The proof of our main theorem (Theorem $\ref{box_deg_less_n/2}$) involves understanding the image of an arc \textbf{u} with respect to $A_0$ and $A_1$. It also turns out that we have to consider various cases based on $|Int(\textbf{u})|$. For a rigorous presentation of the proof we need to calculate the sector numbers of $head(\textbf{u})$, $tail(\textbf{u})$ and their images with respect to $A_{0}$ and $A_{1}$. To avoid clumsiness within the proof, we have collected the required information in the Table 1.\\\\
Let \textbf{u} be an arc such that $|Int(\textbf{u})| \geq 1$. We assume that $m(\textbf{u})=0$. Let p= $ \left \lceil \frac{|Int(\textbf{u})|}{2} \right \rceil$. The first column specifies the condition on $|Int(\textbf{u})|$ and the axis about which the image is taken. The next four columns specify the sector numbers corresponding to head(\textbf{u}), $Im_j(head(\textbf{u}))$, tail(\textbf{u}) and $Im_j(tail(\textbf{u}))$ respectively.\\\\
For the benefit of the reader we explain here how the entries in the first row of the Table are obtained:\\\\
When $|Int(\textbf{u})|$ is even clearly head(\textbf{u}) is the $(p+1)$th sector from
$H_{m(\textbf{u})}=H_0$ in anticlockwise direction, i.e. $S_{p}$. According to Lemma $\ref{Image}$, $Im_{0}(head(\textbf{u}))=S_{-(p+1)}$. Similarly
tail(\textbf{u}) is the $p$th sector from $H_{0}$
in the clockwise direction that is
$tail(\textbf{u})$=$S_{-p}$. According to Lemma $\ref{Image}$ $Im_{0}(tail(\textbf{u}))=S_{p-1}$. The entries in the remaining rows are obtained in a similar way. (We use the notation $len(\textbf{u})$, $len_h(\textbf{u})$ and $len_t(\textbf{u})$ to denote $length(\textbf{u})$, $length(\textbf{u} \cap head(\textbf{u}))$ and $length(\textbf{u} \cap tail(\textbf{u}))$ respectively.)
\begin{table}[!h]\label{sector number}
\caption{Sector numbers corresponding to head(\textbf{u}), tail(\textbf{u}) and their images with respect to median axes for an arc \textbf{u} with m(\textbf{u})=0}
\centering
\begin{tabular}{|p{3cm}|c|c|c|c|}
 \hline
 Condition & \multicolumn{4}{c|}{Sector Number (mod $2\alpha$)} \\
\cline{2-5}
 &head(\textbf{u}) & $Im_j(head(\textbf{u}))$ & tail(\textbf{u}) & $Im_j(tail(\textbf{u}))$\tabularnewline[3pt]
 \hline\hline
 $|Int(\textbf{u})|$ is even and j=0& $S_{p}$ & $S_{-(p+1)}$ & $S_{-p}$ & $S_{p-1}$\tabularnewline[2pt] 
 \hline

 $|Int(\textbf{u})|$ is even and j=1& $S_{p}$ & $S_{-p+1}$ & $S_{-p}$ & $S_{p+1}$\\[2pt]
 \hline\hline

 $|Int(\textbf{u})|$ is odd and j=0 & $S_{p-1}$ & $S_{-p}$ & $S_{-p}$ & $S_{p-1}$\tabularnewline[2pt]
 \hline

 $|Int(\textbf{u})|$ is odd and j=1& $S_{p-1}$ & $S_{-p+2}$ & $S_{-p}$ & $S_{p+1}$\tabularnewline[2pt] 

 \hline
\end{tabular}
\end{table}
\begin{theo}\label{Image_Arc}
Let \textbf{u} and $\textbf{v}$ be arcs such that $\textbf{v} \cap
\textbf{u} = \emptyset$, $len(\textbf{u}) \leq len(\textbf{v})$. Let $|Int(\textbf{u})| \geq 2$ and therefore $p \geq 1$. Let $m(\textbf{u})=0$.
\begin{enumerate}
	\item If $|Int(\textbf{u})|$ is even
		\begin{enumerate}
		\item If $\textbf{v} \cap Im_{0}(\textbf{u}) \neq
		\emptyset$ then $\textbf{v}$ must be anticlockwise adjacent to
		$Im_{0}(\textbf{u})$.
		\item If $\textbf{v} \cap Im_{1}(\textbf{u}) \neq \emptyset$
		then $\textbf{v}$ must be clockwise adjacent to $Im_{1}(\textbf{u})$.
		\end{enumerate}

	\item If $|Int(\textbf{u})|$ is odd and $len_h(\textbf{u}) > len_t(\textbf{u})$
		\begin{enumerate}
		\item If $\textbf{v} \cap Im_{0}(\textbf{u}) \neq
		\emptyset$ then $\textbf{v}$ must be anticlockwise adjacent to
		$Im_{0}(\textbf{u})$.
		\item If $\textbf{v} \cap Im_{1}(\textbf{u}) \neq \emptyset$
		then $\textbf{v}$ must be clockwise adjacent to $Im_{1}(\textbf{u})$.
		\end{enumerate}
	
	\item  If $|Int(\textbf{u})|$ is odd and $len_h(\textbf{u}) = len_t(\textbf{u})$ then $\textbf{v} \cap Im_{0}(\textbf{u}) =
		\emptyset$.
\end{enumerate}
\end{theo}
\noindent{}\textbf{Proof of 1(a):}  Assume that
$|Int(\textbf{u})|=2p$. Since $len(\textbf{v})
\geq len(\textbf{u})$ $=len(Im_{0}(\textbf{u}))$, $\textbf{v} \not \subset
Im_{0}(\textbf{u})$. Now since $\textbf{v} \cap
Im_{0}(\textbf{u}) \neq \emptyset$ it follows that $\textbf{v}$  contains either $l(Im_{0}(\textbf{u}))$ or $r(Im_{0}(\textbf{u}))$ i.e. \textbf{v} is either clockwise or anticlockwise adjacent to $Im_{0}(\textbf{u})$. According to Table 1, $Im_{0}$
$(tail(\textbf{u}))=S_{p-1}$. Note from Table 1 that $tail(\textbf{u})=S_{-p}$ and $head(\textbf{u})=S_{p}$. Therefore $S_{p-1} \subset \textbf{u}$ since $S_{p-1}$ is a sector strictly in-between $S_{p}$ and $S_{-p}$ in the clockwise direction, since $p \geq 1$. So recalling Lemma $\ref{Image2}$ we have $l(Im_{0}(\textbf{u}))) = Im_{0}(r(\textbf{u}))\in Im_{0}(tail(\textbf{u}))=S_{p-1} \subset \textbf{u}$. Therefore $l(Im_{0}(\textbf{u})) \not \in \textbf{v}$ since $\textbf{u} \cap \textbf{v} =\emptyset$. It follows that $r(Im_{0}(\textbf{u})) \in \textbf{v}$ i.e. \textbf{v} is anticlockwise adjacent to $Im_{0}(\textbf{u})$.\\\\
\noindent{}\textbf{Proof of 1(b):} The proof is similar to that of 1(a). According to Table 1,
$Im_{1}(head(\textbf{u}))=S_{h'}$ where
$h'=(-p+1)$. Note that the sector $S_{h'}$ is strictly in-between $head(\textbf{u})=S_{p}$ and $tail(\textbf{u})=S_{-p}$ in the clockwise direction, since $p \geq 1$. So recalling Lemma $\ref{Image2}$ and using Table 1, we have $r(Im_{1}(\textbf{u})))= Im_{1}(l(\textbf{u})) \in Im_{1}(head(\textbf{u}))=S_{h'} \subset \textbf{u}$. Therefore $r(Im_{1}(\textbf{u})) \not \in \textbf{v}$ since $\textbf{u} \cap \textbf{v} =\emptyset$. It follows that $l(Im_{1}(\textbf{u})) \in \textbf{v}$ i.e. \textbf{v} is clockwise adjacent to $Im_{1}(\textbf{u})$.\\\\
\noindent{}\textbf{Proof of 2(a):} Let $|Int(\textbf{u})|=2p+1$ and $len_h(\textbf{u}) > len_t(\textbf{u})$. Then $l(Im_{0}(\textbf{u}))=$ $Im_{0}(r(\textbf{u}))$ $\in$ $Im_0(tail(\textbf{u}))$ $=S_{p-1}$ $=head(\textbf{u})$. Since $len_h(\textbf{u}) > len_t(\textbf{u})$ $=len(Im_0(\textbf{u} \cap tail(\textbf{u})))$. We infer that $l(Im_{0}(\textbf{u})) \in \textbf{u}$ and therefore $l(Im_{0}(\textbf{u})) \not \in \textbf{v}$. Therefore $r(Im_{0}(\textbf{u})) \in \textbf{v}$ i.e. $\textbf{v}$ is anticlockwise adjacent to $Im_{0}(\textbf{u})$.\\\\
\noindent{}\textbf{Proof of 2(b):} The proof is similar to that of 1(a). According to Table 1,
$Im_{1}(head(\textbf{u}))$ $=S_{h'}$ where
$h'=(-p+2)$. Since $head(\textbf{u})=S_{p-1}$ and $tail(\textbf{u})=S_{-p}$, $S_{h'}$ is strictly in-between $head(\textbf{u})$ and $tail(\textbf{u})$ in the clockwise direction since $p \geq 2$. So recalling Lemma $\ref{Image2}$ we have $r(Im_{1}(\textbf{u}))= Im_{1}(l(\textbf{u})) \in Im_{1}(head(\textbf{u}))=S_{h'} \subset \textbf{u}$. Therefore $r(Im_{1}(\textbf{u})) \not \in \textbf{v}$ since $\textbf{u} \cap \textbf{v} =\emptyset$. It follows that $l(Im_{1}(\textbf{u})) \in \textbf{v}$ i.e. \textbf{v} is clockwise adjacent to $Im_{1}(\textbf{u})$.\\\\
\noindent{}\textbf{Proof of 3:} When $|Int(\textbf{u})|$ is odd then $Im_0(head(\textbf{u}))$ $=S_{-p}$ $=tail(\textbf{u})$ and $Im_0(tail(\textbf{u}))$ $=S_{p-1}$ $=head(\textbf{u})$. Now since $len_h(\textbf{u}) = len_t(\textbf{u})$, we have $r(\textbf{u})= Im_{0}(l(\textbf{u}))$ and $l(\textbf{u})= Im_{0}(r(\textbf{u}))$. Hence $Im_{0}(\textbf{u})=\textbf{u}$ and therefore $\textbf{v} \cap Im_{0}(\textbf{u}) = \emptyset$ since $\textbf{u} \cap \textbf{v} = \emptyset$.\\
\begin{theo}\label{Image_Arc2}
Let \textbf{u} and $\textbf{v}$ be arcs such that $\textbf{v} \cap
\textbf{u} = \emptyset$, $len(\textbf{u}) \leq len(\textbf{v})$. Let $|Int(\textbf{u})| = 1$ and therefore $p=1$. Let $m(\textbf{u})=0$.
\begin{enumerate}
 \item If $len_h(\textbf{u}) > len_t(\textbf{u})$
 	\begin{enumerate}
 	\item If $\textbf{v} \cap Im_{0}(\textbf{u}) \neq \emptyset$ then $\textbf{v}$ must be anticlockwise adjacent to $Im_{0}(\textbf{u})$.
 	\item If $\textbf{v} \cap Im_{0}(\textbf{u}) \neq \emptyset$ and $\textbf{v} \cap Im_{1}(\textbf{u}) \neq \emptyset$ then $\textbf{v}$ must be clockwise adjacent to $Im_{1}(\textbf{u})$.
 	\end{enumerate}
\item  If $len_h(\textbf{u}) = len_t(\textbf{u})$ then $\textbf{v} \cap Im_{0}(\textbf{u}) =
		\emptyset$.
\end{enumerate}
\end{theo}
\noindent{}\textbf{Proof of 1(a):} When $len_h(\textbf{u}) > len_t(\textbf{u})$ then $l(Im_{0}(\textbf{u})))=$ $Im_{0}(r(\textbf{u}))$ $\in$ $Im_0(tail(\textbf{u}))$ $=S_{0}$ $=head(\textbf{u})$. Since $len_h(\textbf{u}) > len_t(\textbf{u})$ $=len(Im_0(\textbf{u} \cap tail(\textbf{u})))$. We infer that $l(Im_{0}(\textbf{u})) \in \textbf{u}$ and therefore $l(Im_{0}(\textbf{u})) \not \in \textbf{v}$. Therefore $r(Im_{0}(\textbf{u})) \in \textbf{v}$ i.e. $\textbf{v}$ is anticlockwise adjacent to $Im_{0}(\textbf{u})$.\\\\
\noindent{}\textbf{Proof of 1(b):} According to Table 1,
$Im_{1}(head(\textbf{u}))=S_{h'}$ where
$h'=(-p+2)$ $=1$ and $Im_{1}(tail(\textbf{u}))=S_{h''}$ where
$h''=p+1$ $=2$. So $Im_{1}(\textbf{u})$ goes from $S_{2}$ to $S_{1}$ in clockwise direction. Since $len(\textbf{v})
\geq len(\textbf{u})$ $=len(Im_{1}(\textbf{u}))$, $\textbf{v} \not \subset
Im_{1}(\textbf{u})$. Therefore \textbf{v} is either clockwise or anticlockwise adjacent to $Im_{1}(\textbf{u})$. If possible let $l(Im_{1}(\textbf{u})) \not \in \textbf{v}$. But since $\textbf{v} \cap Im_{0}(\textbf{u}) \neq \emptyset$, by case 1(a), $r(Im_0(\textbf{u})) \in \textbf{v}$ $\implies$ $Im_0(l(\textbf{u})) \in \textbf{v}$ $\implies$ $(\textbf{v} \cap Im_{0}(head(\textbf{u}))) \neq \emptyset$ $\implies$ $(\textbf{v} \cap S_{-1}) \neq \emptyset$. It follows that \textbf{v} intersects all the sectors from $S_{1}$ to $S_{-1}$ in clockwise direction. (Note that this cannot be in anticlockwise direction since in that case $l(Im_{1}(\textbf{u})) \in \textbf{v}$). Then $S_{0} \subset \textbf{v}$. But it is not possible since $S_{0} \cap \textbf{u} \neq \emptyset$ and $\textbf{u} \cap \textbf{v}=\emptyset$. Therefore $\textbf{v}$ is clockwise adjacent to $Im_{1}(\textbf{u})$.\\\\
\noindent{}\textbf{Proof of 2:} When $len_h(\textbf{u}) = len_t(\textbf{u})$ then $Im_0(head(\textbf{u}))$ $=S_{-1}$ $=tail(\textbf{u})$ and $Im_0(tail(\textbf{u}))$ $=S_{0}$ $=head(\textbf{u})$. Now since $len_h(\textbf{u}) = len_t(\textbf{u})$, we have $r(\textbf{u})= Im_{0}(l(\textbf{u}))$ and $l(\textbf{u})= Im_{0}(r(\textbf{u}))$. Hence $Im_{0}(\textbf{u})=\textbf{u}$ and therefore $\textbf{v} \cap Im_{0}(\textbf{u}) = \emptyset$ since $\textbf{u} \cap \textbf{v} = \emptyset$.\\

\subsection{Construction of interval graphs}
Let $G$ be a circular arc graph with maximum degree $\Delta < \left \lfloor \frac{n(\alpha-1)}{2\alpha}\right \rfloor$. In this section we will construct
$\alpha$ interval graphs $I_0, I_1, \ldots, I_{\alpha-1}$ such that
$G=I_0 \cap I_1 \cap \ldots \cap I_{\alpha-1}$. Now consider a system of reference axes based on $\alpha$. To define $I_j$ we
map each vertex $v \in V$ to an interval on $A_j$ by the mapping
$\forall u \in V, \ g_j(u)=[\inf Proj_j(\textbf{u}), \ \sup
Proj_j(\textbf{u})]$.
\begin{lemma}\label{supgraph}
For each interval graph $I_j$, $0 \leq j < \alpha$, $E(G) \subseteq
E(I_j)$.
\end{lemma}
\begin{proof}
Let $(u,v) \in E(G)$. There is a point $P \in \textbf{u} \cap \textbf{v}$. Clearly $Proj_j(P)
\in g_j(u) \cap g_j(v)$ for $0 \leq j < \alpha$. So, $E(G) \subseteq
E(I_j)$ for $0 \leq j < \alpha$.
\end{proof}
\begin{lemma}\label{missing_edge}
If $(u,v) \not \in E(G)$ then there exists some $j$, $0 \leq j <
\alpha$, such that $(u, v) \notin E(I_j)$.
\end{lemma}
\begin{proof}
Since $(u,v) \not \in E(G)$, we have $\textbf{u} \cap \textbf{v} =
\emptyset$. Without loss of generality we can assume that
$len(\textbf{u}) \leq len(\textbf{v})$. We consider the following cases\\\\
$\noindent\textbf{Case 1:}$ ${\it Int(\textbf{u})=\emptyset}$. Then \textbf{u} is properly contained in some sector. Let
$\textbf{u} \subset S_k$. $A_{k_l}$ and $A_{k_r}$ be the left and
right reference axes of $S_k$ where $k_l=(k+1) \ mod \ \alpha$ and
$k_r=k  \ mod \ \alpha$. In this case we will show that either
$(u,v) \not \in E(I_{k_l})$ or $(u,v) \not \in E(I_{k_r})$. Suppose not.
Then $g_{k_l}(u)$ $\cap$ $g_{k_l}(v)$ $\neq$ $\emptyset$ and
therefore there exists some point $P_u$ $\in$ $\textbf{u}$ and some
point $P_v$ $\in$ $\textbf{v}$ such that
$Proj_{k_l}(P_u)=Proj_{k_l}(P_v)$. Then either $P_u=P_v$ or
$Im_{k_l}(P_u)=P_v$ by fact $\ref{projpoint}$. But as $\textbf{u}
\cap \textbf{v} = \emptyset$ we have $P_u \neq P_v$. Hence according
to Lemma $\ref{proj}$, $P_v$ must belong to the sector $(2((k+1) \ mod \
\alpha)-k-1) \ mod \ 2\alpha=$ $(k+1) \ mod \ 2\alpha$. Similarly
if $g_{k_r}(u)$ $\cap$ $g_{k_r}(v)$ $\neq$ $\emptyset$ then there
exists some point $Q_u$ $\in$ $\textbf{u}$ and some point $Q_v$
$\in$ $\textbf{v}$ such that $Proj_{k_r}(Q_u)=Proj_{k_r}(Q_v)$. For
similar reasons as above $Q_v$ must belong to the sector $((k-1) \
mod \ 2\alpha)$. Thus \textbf{v} intersects the sectors $S_{(k-1) \
mod \ 2\alpha}$ as well as $S_{(k+1) \ mod \ 2\alpha}$. But as
$\textbf{v}$ does not intersect $\textbf{u}$, $\textbf{v}$ cannot
pass through $S_k$. It follows that it intersects all the other
(2$\alpha$-1) sectors except $S_k$. In particular \textbf{v}
properly contains all the $(2\alpha-3)$ sectors other than $S_{(k-1)
\ mod \ 2\alpha}$, $S_k$ and $S_{(k+1) \ mod \ 2\alpha}$ and therefore $len(\textbf{v})$ $\geq$ $(2\alpha-3)\times sector \ length$ $=\frac{2\Pi(2\alpha-3)}{2\alpha}$ $\geq \frac{\Pi(\alpha-1)}{\alpha}$. Since $\alpha \geq 2$ this contradicts Observation 1, part 1.\\\\
$\noindent\textbf{Case 2:}$ ${\it When \ Int(\textbf{u}) \neq \emptyset \ and \ |Int(\textbf{u})| \ is \ even }$. If $\alpha=2$ then $len(\textbf{u}) < \frac{\Pi(\alpha-1)}{\alpha}=\frac{\Pi}{2}$ and therefore \textbf{u} cannot properly contain a sector, since a sector has length $\frac{\Pi}{2}$ in this case. Then $|Int(\textbf{u})| \leq 1$ contradicts the assumption. Therefore we can assume that $\alpha \geq 3$. Let $w=m(\textbf{u}) \ mod \ \alpha$ and $w'=(m(\textbf{u})+1) \ mod \ \alpha$, where $m(\textbf{u})$ is the median axis number of \textbf{u}. Our intention is to show that either $(u,v) \not \in
E(I_{w})$ or $(u,v) \not \in E(I_{w'})$. Now for the ease of presentation we will renumber our system of reference axes such that $m(\textbf{u})=0$. Note that it only corresponds to a rotation of the circular arc system in such a way that $H_{m(\textbf{u})}$ comes to $H_0$. After this transformation the interval graphs $I_w$ and $I_{w'}$ will also get \textit{renamed} to $I_0$ and $I_1$ respectively. So now we have only to prove that either $(u,v) \not \in
E(I_{0})$ or $(u,v) \not \in E(I_{1})$. Let $|Int(\textbf{u})|=2p$. Let $P^+$ and $P^-$ be the intersection point of the circle with $H_0$ and $H_{\alpha}$ (i.e. positive and negative x-axis) respectively. If $g_{0}(u)$ $\cap$
$g_{0}(v)$ $\neq$ $\emptyset$ then $\textbf{v} \cap
Im_{0}(\textbf{u}) \neq \emptyset$. Since $\textbf{v}
\cap \textbf{u} =\emptyset$ according to Theorem
$\ref{Image_Arc}$ part 1(a), \textbf{v} must be anticlockwise adjacent to
$Im_{0}(\textbf{u})$. Therefore
 $r(Im_{0}(\textbf{u})) \in \textbf{v}$ which implies
 that $Im_{0}(l(\textbf{u})) \in \textbf{v}$ by Lemma
$\ref{Image2}$. Similarly since $g_{1}(u)$ $\cap$
$g_{1}(v)$ $\neq$ $\emptyset$, according to Theorem
$\ref{Image_Arc}$ part 1(b), \textbf{v} must be clockwise adjacent to
$Im_{1}(\textbf{u})$. Therefore
 $l(Im_{1}(\textbf{u})) \in \textbf{v}$ which implies
 that $Im_{1}(r(\textbf{u})) \in \textbf{v}$ by Lemma
$\ref{Image2}$. But according to Table 1, $Im_{1}(r(\textbf{u})) \in$ $Im_{1}$ $(tail(\textbf{u}))$ $=
S_{p+1}$ and $Im_{0}(l(\textbf{u}))$ $\in$ $Im_{0}(head(\textbf{u}))$ $=
S_{-(p+1)}$. Therefore \textbf{v} intersects the sectors
$S_{-(p+1)}$ as well as $S_{p+1}$. Now since $(p+1) \leq (\alpha-1)$ for $\alpha \geq 3$ (see Observation 1, part 1), the sector $S_{p+1}$ belongs to the half circle above $A_0$ (i.e. x-axis). Similarly the sector $S_{-(p+1)}$ belongs to the half circle below $A_0$ (i.e. x-axis). It follows that the arc \textbf{v} should intersect $A_0$ either at $P^+$ or $P^-$. But in view of definition of median axis $P^+ \in \textbf{u}$. Since $\textbf{u} \cap \textbf{v} = \emptyset$ we have $P^- \in \textbf{v}$.
\begin{align}\label{v_sectors}
 So \ we \ infer \ that \ \textbf{v} \ intersects \ all \ the \ sectors \ from \ S_{-(p+1)} \ to \ S_{p+1} \ in \ the \ clockwise \notag \\ direction.
\end{align}
  Moreover we can infer that \textbf{v} properly contains all the sectors strictly in-between $S_{-(p+1)}$ and $S_{p+1}$ in the clockwise direction (if there is any such sector). Note that since $head(\textbf{u})=S_{p}$ and $tail(\textbf{u})=S_{-p}$, \textbf{u} properly contains all the sectors strictly in-between $S_{p}$ and $S_{-p}$ in the clockwise direction (if there is any such sector).
\begin{align}\label{coverage_even}
Thus \ \textbf{v} \cup \textbf{u} \ properly \ contains \ all \ the \ (2\alpha-4) \ sectors \ other \ than \ S_{p}, \ S_{-p}, \ S_{-(p+1)} \notag \\ and \ S_{p+1} \ namely \ head(\textbf{u}), \ tail(\textbf{u}), \ Im_{0}(head(\textbf{u})) \ and \ Im_{1}(tail(\textbf{u})).
\end{align}
\noindent{\textbf{Claim 1:}} If $|Int(\textbf{u})|$ is even
\begin{enumerate}
 \item ($len_h(\textbf{u}$)+len($\textbf{v} \cap S_{-(p+1)}$)) $\geq len(S_{-(p+1)})$.
 \item ($len_t(\textbf{u}$)+len($\textbf{v} \cap S_{p+1}$)) $\geq len(S_{p+1})$.
\end{enumerate}
\begin{proof}
Since $len_h(\textbf{u})$ $=len(\textbf{u}$ $\cap$ $head(\textbf{u}))$ $=len(Im_0(\textbf{u})$ $\cap$ $Im_0(head(\textbf{u})))$ and recalling from Table 1 that $Im_0(head(\textbf{u}))$ $=S_{-(p+1)}$ it is sufficient to show that $len(Im_0(\textbf{u})$ $\cap$ $S_{-(p+1)})$ $+len(\textbf{v}$ $\cap$ $S_{-(p+1)})$ $\geq$ $len(S_{-(p+1)})$. Let $P=Im_{0}(l(\textbf{u}))$ $=r(Im_{0}(\textbf{u}))$. We know that $P$ $\in$ $Im_{0}$ $(head(\textbf{u}))$ $=S_{-(p+1)}$. Let $P_L=l(S_{-(p+1)})$ and $P_R=$ $r(S_{-(p+1)})$. Note that $head(\textbf{u}) \cap \textbf{u}$ is an arc from l(\textbf{u}) to r(head(\textbf{u})) in the clockwise direction. Therefore $Im_{0}(\textbf{u} \cap head(\textbf{u}))$ $=Im_{0}(\textbf{u}) \cap Im_{0}(head(\textbf{u}))$ $= (Im_{0}(\textbf{u})$ $\cap$ $S_{-(p+1)})$ is the arc from $Im_{0}(l(\textbf{u}))$ $=r(Im_{0}$ $(\textbf{u}))=P$ to $Im_{0}(r(head(\textbf{u})))$ $=$ $l(Im_{0}$ $(head(\textbf{u})))$ $=$ $P_L$ within the sector $Im_{0}(head(\textbf{u}))$ $=S_{-(p+1)}$. Since $g_0(u) \cap g_0(v) \neq \emptyset$ by Theorem $\ref{Image_Arc}$ part 1(a), $r(Im_{0}(\textbf{u}))$ $=P \in \textbf{v}$. Moreover by statement $\eqref{v_sectors}$, \textbf{v} intersects all the sectors from $S_{-(p+1)}$ to $S_{p+1}$ in the clockwise direction. It follows that $P_R \in \textbf{v}$. In other words the arc from P to $P_R$ (in the clockwise direction) is contained in \textbf{v}. Thus we have ($len_h(\textbf{u}$)+len($\textbf{v} \cap S_{-(p+1)}$)) $\geq len(S_{-(p+1)})$.\\
 Proof of part 2 is similar as that of the part 1 except that instead of head(\textbf{u}) we have to use tail(\textbf{u}). Moreover instead of reference axis $A_{0}$ we have to use $A_{1}$ and sector $S_{p+1}$ should replace $S_{-(p+1)}$.
\end{proof}
Thus in view of statement $\eqref{coverage_even}$ and claim 1 and recalling that the length of one sector is $\frac{\Pi}{\alpha}$ we can say that len(\textbf{u}) + len(\textbf{v}) $\geq$
$(2\alpha-2)\frac{\Pi}{\alpha}$. Now since $len(\textbf{v}) \geq len(\textbf{u})$ we get $len(\textbf{v}) \geq
(\alpha-1)\frac{\Pi}{\alpha}$, a contradiction to Observation 1, part 1.\\\\
$\noindent\textbf{Case 3:}$ {\it If $|Int(\textbf{u})| \geq 1$ and $|Int(\textbf{u})|$ is odd.} As we did in case 2, we first renumber our system of reference axes such that $m(\textbf{u})=0$, i.e. we rotate the circular arc system in such a way that $H_{m(\textbf{u})}$ comes to $H_0$. Note that after the rotation $H_{m(\textbf{u})+1}$ will correspond to $H_1$ and $H_{m(\textbf{u})-1}$ will correspond to $H_{-1}$. We consider the following subcases.\\\\
$\noindent\textbf{Subcase 3.1:}$ {\it$len_h(\textbf{u})$ $=$
$len_t(\textbf{u})$}. By Theorem $\ref{Image_Arc}$ part 3, $Im_{0}(\textbf{u})$  =  \textbf{u} and hence
$g_{0}(u)$ $\cap$ $g_{0}(v)$ $=
\emptyset$. Therefore we have $(u,v) \not \in E(I_{0})$ as required.\\\\
$\noindent\textbf{Subcase 3.2:}$ {\it $len_h(\textbf{u})$ $\neq$ $len_t(\textbf{u})$}. Our proof requires that we consider two cases namely
\begin{enumerate}
 \item $len_h(\textbf{u})$ $>$ $len_t(\textbf{u})$
 \item $len_h(\textbf{u})$ $<$ $len_t(\textbf{u})$
\end{enumerate}
In the first case our intention is to show that either $(u,v) \not \in
E(I_{m(\textbf{u}) \ mod \ \alpha})$ or $(u,v) \not \in E(I_{(m(\textbf{u})+1) \ mod \ \alpha})$ and in the second case our intention is to show that either $(u,v) \not \in
E(I_{m(\textbf{u}) \ mod \ \alpha})$ or $(u,v) \not \in E(I_{(m(\textbf{u})-1) \ mod \ \alpha})$. In fact there is no need to analyze these two cases separately. If $len_h(\textbf{u})$ $<$ $len_t(\textbf{u})$ then we can do a further transformation to our circular arc system namely a reflection about the x-axis. Note that after the reflection the left and right endpoints of each arc gets interchanged (see Lemma $\ref{Image2}$). Therefore in this transformed circular arc system $len_h(\textbf{u})$ $>$ $len_t(\textbf{u})$. Also note that the numbering of the reference axes has also changed due to reflection: $A_0$ still corresponds to $A_0$ but $A_1$ corresponds to $A_{-1}$ and $A_{-1}$ corresponds to $A_1$. Thus the half axis that was originally numbered $m(\textbf{u})$ will be numbered 0 and $(m(\textbf{u})-1)$ will be numbered 1 after the rotation and reflection. Note that the interval graphs are also renamed after this transformation. Thus in both cases we need only to prove that either $(u,v) \not \in
E(I_{0})$ or $(u,v) \not \in E(I_{1})$. Suppose not. Then we claim the following\\\\
\noindent{}\textbf{Claim 2:} $\textbf{v}$ $\cup$ $\textbf{u}$ properly contains all the $(2\alpha-4)$ sectors other than head(\textbf{u}),  tail(\textbf{u}), $S_{p}$ and $S_{p+1}$.
\begin{proof}
 If $\alpha=2$, $(2\alpha-4)=0$ and the claim is trivially true. Thus we can assume $\alpha \geq 3$. Let $P^+$ and $P^-$ be the intersection point of the circle with $H_0$ and $H_{\alpha}$ (i.e. positive and negative x-axis) respectively. If $g_{0}(u)$ $\cap$
$g_{0}(v)$ $\neq$ $\emptyset$ then $\textbf{v} \cap
Im_{0}(\textbf{u}) \neq \emptyset$. Since $\textbf{v}
\cap \textbf{u} =\emptyset$, using Theorem $\ref{Image_Arc}$ part 2(a) and Theorem $\ref{Image_Arc2}$ part 1(a), we can conclude that \textbf{v} must be anticlockwise adjacent to
$Im_{0}(\textbf{u})$. So according to Lemma
$\ref{Image2}$, $Im_{0}(l(\textbf{u})) \in \textbf{v}$.
Similarly since $g_{1}(u)$ $\cap$
$g_{1}(v)$ $\neq$ $\emptyset$, by Theorem $\ref{Image_Arc}$ part 2(b) and Theorem $\ref{Image_Arc2}$ part 1(b) \textbf{v} must be clockwise adjacent to
$Im_{1}(\textbf{u})$. Hence according to
Lemma $\ref{Image2}$ $Im_{1}(r(\textbf{u})) \in
\textbf{v}$. Thus according to Table 1, \textbf{v} intersects the sectors
$S_{-p}$ as well as $S_{p+1}$. Now since $(p+1) \leq (\alpha-1)$ for $\alpha \geq 3$ (see Observation 1 part 2), the sector $S_{p+1}$ belongs to the half circle above $A_0$ (i.e. x-axis). Similarly the sector $S_{-p}$ belongs to the half circle below $A_0$ (i.e. x-axis). It follows that the arc \textbf{v} should intersect $A_0$ either at $P^+$ or $P^-$. But in view of definition of median axis $P^+ \in \textbf{u}$. Since $\textbf{u} \cap \textbf{v} = \emptyset$ we have $P^- \in \textbf{v}$. So we infer the following
\begin{align}\label{v_sectors1}
 If \ \alpha \geq 3 \ then \ \textbf{v} \ intersects \ all \ the \ sectors \ from \ S_{-p} \ to \ S_{p+1} \ in \ the \ clockwise \notag \\ direction.
\end{align}
Moreover we can infer that \textbf{v} properly contains all the sectors strictly in-between $S_{-p}$ and $S_{p+1}$ in clockwise direction (if there is any such sector). Also note that since $head(\textbf{u})=S_{p-1}$ and $tail(\textbf{u})=S_{-p}$, \textbf{u} properly contains all the sectors strictly  in-between $S_{p-1}$ to $S_{-p}$ in the clockwise direction (if there is any such sector). Therefore $\textbf{v}$ $\cup$ $\textbf{u}$ properly contains all the $(2\alpha-4)$ sectors other than $S_{p-1}$, $S_{-p}$, $S_{p}$ and $S_{p+1}$ namely head(\textbf{u}),  tail(\textbf{u}), $S_{p}$ and $S_{p+1}$.
\end{proof}
\noindent{}\textbf{Observation 2:} Note that statement (3) is valid even when $\alpha=2$. When $\alpha=2$ and $|Int(\textbf{u})|=1$, $p=1$. Since $g_{0}(u)$ $\cap$
$g_{0}(v)$ $\neq$ $\emptyset$ and $r(Im_0(\textbf{u})) \in S_{-1}$, $\textbf{v} \cap S_{-1} \neq \emptyset$. Similarly since $g_{1}(u)$ $\cap$
$g_{1}(v)$ $\neq$ $\emptyset$ and $l(Im_1(\textbf{u})) \in S_{2}$, $\textbf{v} \cap S_{2} \neq \emptyset$. Therefore \textbf{v} goes from $S_{-1}$ to $S_{2}$ in clockwise direction because otherwise $S_0 \subset \textbf{v}$, a contradiction to our assumption that $\textbf{u} \cap \textbf{v}=\emptyset$.\\\\
\noindent{\textbf{Claim 3:}} If $|Int(\textbf{u})|$ is odd and $len_h(\textbf{u})$ $>$ $len_t(\textbf{u})$
\begin{enumerate}
 \item ($len_h(\textbf{u}$)+len($\textbf{v} \cap S_{-p}$)) $\geq len(S_{-p})$.
 \item ($len_t(\textbf{u}$)+len($\textbf{v} \cap S_{p+1}$)) $\geq len(S_{p+1})$.
\end{enumerate}
\begin{proof}
Since $len_h(\textbf{u})$ $=len(\textbf{u}$ $\cap$ $head(\textbf{u}))$ $=len(Im_0(\textbf{u})$ $\cap$ $Im_0(head(\textbf{u})))$ and recalling from Table 1 that $Im_0(head(\textbf{u}))$ $=S_{-p}$ it is sufficient to show that $len(Im_0(\textbf{u})$ $\cap$ $S_{-p})$ $+len(\textbf{v}$ $\cap$ $S_{-p})$ $\geq$ $len(S_{-p})$. Let $P=Im_{0}(l(\textbf{u}))$ $=r(Im_{0}(\textbf{u}))$. We know that $P$ $\in$ $Im_{0}$ $(head(\textbf{u}))$ $=S_{-p}$. Let $P_L=l(S_{-p})$ and $P_R=$ $r(S_{-p})$. Note that $head(\textbf{u}) \cap \textbf{u}$ is an arc from l(\textbf{u}) to r(head(\textbf{u})) in the clockwise direction. Therefore $Im_{0}(\textbf{u} \cap head(\textbf{u}))$ $=Im_{0}(\textbf{u}) \cap Im_{0}(head(\textbf{u}))$ $= (Im_{0}(\textbf{u})$ $\cap$ $S_{-p})$ is the arc from $Im_{0}(l(\textbf{u}))$ $=r(Im_{0}$ $(\textbf{u}))=P$ to $Im_{0}(r(head(\textbf{u})))$ $=$ $l(Im_{0}$ $(head(\textbf{u})))$ $=$ $P_L$ within the sector $Im_{0}(head(\textbf{u}))$ $=S_{-p}$. Since $g_0(u) \cap g_0(v) \neq \emptyset$ by Theorem $\ref{Image_Arc}$ part 2(a) and Theorem $\ref{Image_Arc2}$ part 1(a), $r(Im_{0}(\textbf{u}))$ $=P \in \textbf{v}$. Moreover by statement $\eqref{v_sectors1}$ and Observation 2, \textbf{v} intersects all the sectors from $S_{-p}$ to $S_{p+1}$ in the clockwise direction. It follows that $P_R \in \textbf{v}$. In other words the arc from P to $P_R$ (in the clockwise direction) is contained in \textbf{v}. Thus we have ($len_h(\textbf{u}$)+len($\textbf{v} \cap S_{-p}$)) $\geq len(S_{-p})$.\\
 Proof of part(2) is the same as that of the part(1) except that instead of head(\textbf{u}) we have to use tail(\textbf{u}). Moreover instead of reference axis $A_{0}$ we have to use $A_{1}$ and sector $S_{p+1}$ should replace $S_{-p}$.\\\\
\end{proof}
Now in view of claim 2, claim 3 and recalling that the length of one sector is $\frac{\Pi}{\alpha}$ we can say that len(\textbf{u}) + len(\textbf{v}) $\geq$
$(2\alpha-2)\frac{\Pi}{\alpha}$. Now since $len(\textbf{v}) \geq len(\textbf{u})$ we get $len(\textbf{v}) \geq
(\alpha-1)\frac{\Pi}{\alpha}$, a contradiction to Observation 1, part 1.
\end{proof}
\noindent{}Combining Lemma $\ref{supgraph}$ and $\ref{missing_edge}$ we have the following Theorem.
\begin{theo}\label{box_deg_less_n/2}
For a circular arc graph $G$, $box(G) \leq \alpha$ when $\Delta < \left \lfloor
\frac{n(\alpha-1)}{2\alpha} \right \rfloor$.
\end{theo}
\subsection{Tightness Result}
It would have been nice if we could show a circular arc graph $G$
with $\Delta = \frac{n(\alpha-1)}{2\alpha}$ and $box(G) > \alpha$ in
order to demonstrate the tightness of the bound given in Theorem
$\ref{box_deg_less_n/2}$. Unfortunately we are unable to construct
such a graph. In this section we will show a circular arc graph $G$
with $box(G) > \alpha$ and $\Delta =
n\frac{(\alpha-1)}{2\alpha}+\frac{n}{2\alpha(\alpha+1)}+(\alpha+2)$.
Thus the upper bound on $\Delta$ cannot be raised by more than an
additive factor of $\frac{n}{2\alpha(\alpha+1)}+(\alpha+2)$ $\approx
\frac{n}{2\alpha^2}$ keeping $box(G) \leq \alpha$. We describe the graph $G$ below.\\
Let $V(G)$ be the disjoint union of 3 sets
of vertices namely $S_1,S_2$ and $S_3$ such that $|S_1|=2\alpha+2$, $|S_2|=2\alpha+2$ and $|S_3|=n-(4\alpha+4)$. (We will assume that n is divisible by $2\alpha+2$). Let $S_1=\{ u_0,u_1, \ldots,u_{2\alpha+1} \}$, $S_2=\{ v_0,v_1, \ldots,v_{2\alpha+1} \}$. Let $S_3$ contain the remaining $n-(4\alpha+4)$ vertices. We will partition the set $S_3$ further into $2\alpha+2$ subsets $H_0,H_1, \ldots,H_{2\alpha+1}$ where $|H_i|=\frac{n}{2\alpha+2}-2$ i.e. $S_3=H_0 \cup H_1 \cup \ldots \cup H_{2\alpha+1}$. The circular arc graph $G$ on the vertex set $V(G)=S_1 \cup S_2 \cup S_3$ is constructed as follows.\\
Let us take a circle. Let t=$2\alpha+2$. Place t points on the circle such that the distance between any adjacent pair of points is equal. Let these points be $p_0,p_1,\ldots,p_{t-1}$ in clockwise direction. For each i, $0 \leq i \leq (2\alpha+1)$, $u_i$ is mapped to the circular arc from $p_i$ to $p_{(i+\alpha) \ mod \ t}$  in the clockwise direction. Similarly for each i, $0 \leq i \leq (2\alpha+1)$, $v_i$ is mapped to the circular arc from $p_i$ to $p_{(i+1) \ mod \ t}$  in the clockwise direction. For each i, $0 \leq i \leq (2\alpha+1)$, the vertices in $H_i$ are mapped to arcs of sufficiently small length placed strictly in-between $p_i$ and $p_{(i+1) \ mod \ t}$  (i.e. strictly inside the arc corresponding to $v_i$) such that arcs in $H_i$ are pairwise non-intersecting.\\

\noindent $\textbf{Claim 1:}$ Maximum degree of $G$ is
$\frac{n\alpha}{2(\alpha+1)}+(\alpha+2)$=
$\frac{n(\alpha-1)}{2\alpha}+\frac{n}{2\alpha(\alpha+1)}+(\alpha+2)$.
\begin{proof}
Any vertex in set $S_1$ is connected to 2$\alpha$ vertices in set
$S_1$, ($\alpha$+2) vertices in set $S_2$ and $\alpha(\frac{n}{2(\alpha+1)}-2)$
vertices in set $S_3$. So maximum degree of any vertex in $S_1$ is
$\frac{n\alpha}{2(\alpha+1)}+(\alpha+2)$.\\
Any vertex in set $S_2$ is connected to ($\alpha$+2) vertices in set $S_1$, 2 vertices in set $S_2$
and $(\frac{n}{2(\alpha+1)}-2)$ vertices in set $S_3$. So maximum degree of
any vertex in $S_2$ is $\frac{n}{2(\alpha+1)}+\alpha+2$.\\
Any vertex in set $S_3$ is connected to ($\alpha$-1) vertices in set $S_1$
and 1 vertex in set $S_2$. So maximum degree of any vertex in $S_3$
is $\alpha$.\\
So maximum degree of $G$ is $\frac{n\alpha}{2(\alpha+1)}+(\alpha+2)$.
\end{proof}
\noindent $\textbf{Claim 2:}$ $box(G) > \alpha$.
\begin{proof}
It is easy to see that in $G$, $S_1$ induces a subgraph isomorphic to $\overline{(\alpha+1)K_2}$. (See the circular arc construction of $\overline{(\frac{n}{2})K_2}$ given in the introduction). From $\cite{Roberts}$, box($\overline{(\alpha+1)K_2}$) $=
\alpha+1$ and therefore box(G) $> \alpha$.
\end{proof}
So there exists a graph on n vertices with maximum degree
$\frac{n\alpha}{2(\alpha+1)}+(\alpha+2)$ and boxicity $> \alpha$.
\section{An upper bound based on minimum overlap set}\label{GeneralConstruction}
Cardinality of minimum overlap set is a parameter (see definition $\ref{Minimum overlap point}$) one encounters frequently in the study of circular arc graphs. For example let $F$ be the family of arcs in a circular arc graph $G$ then $\chi(F) \leq r_{sup}(F) +r_{inf}(F)$ $\cite{Hadwiger}$. In the following theorem we relate $box(G)$ with $r_{inf}(G)$. 
\begin{theo}\label{box_general}
For a circular arc graph $G$, $box(G) \leq r_{inf}(G) + 1$ where
$r_{inf}(G)$ is the cardinality of minimum overlap set with respect to any circular arc representation of $G$.
\end{theo}
\noindent {}From now on we will use $r_{inf}$ instead of $r_{inf}(G)$. Let $\mathcal{O}(P_{inf})$= $\{\textbf{w}_1,\textbf{w}_2,\ldots,\textbf{w}_{r_{inf}} \}$. We construct $r_{inf}+1$ interval super graphs of $G$ say
$I_1, I_2,\ldots, I_{r_{inf}+1}$ as follows.\\

\noindent \textbf{Construction of $I_i$ for $1 \leq i \leq
r_{inf}$:} For each $\textbf{w}_i$ we will construct an interval graph $I_i$. To construct $I_i$
we map each vertex $v \in V(G)$ to an interval on the real line by the mapping\\
\vspace{-.6cm}
\begin{eqnarray*}
  g_i(v) &=& [0, 1] \ \ \ \ \ \ \ \ \ \ if \ v = w_i  \\
  	 &=& [1, 2] \ \ \ \ \ \ \ \ \ \ if \ v \in N_G(w_i).\\
  	 &=& [2, 3] \ \ \ \ \ \ \ \ \ \ if \ v \in V(G) - (N_G(w_i) \cup \{w_i\}).
\end{eqnarray*}

\noindent{\bf Construction of $I_{r_{inf}+1}$}: Recall that according to definition $\ref{clockwise ordering}$,
$\sigma$ is a clockwise ordering starting from $P_{inf}$. So
$\sigma(l(\textbf{v})) < \sigma(r(\textbf{v}))$ when $\textbf{v}
\not \in \mathcal{O}(P_{inf})$. We define the mapping as follows:

\vspace{-.6cm}

\begin{eqnarray*}
  g_{r_{inf}+1}(v) &=& [1, 2n] \ \ \ \ \ \ \ \ \ if \ \textbf{v} \in \mathcal{O}(P_{inf}). \\
             &=& [\sigma(l(\textbf{v})),\sigma(r(\textbf{v}))] \ \ \ \ \   otherwise. \\
\end{eqnarray*}

%\\\\
%
%
\begin{lemma}
For each interval graph $I_j$ where $1 \leq j \leq r_{inf}+1$, $E(G)
\subseteq E(I_j)$.
\end{lemma}

\begin{proof}
We consider the following two cases\\
\noindent{\bf Case 1:} When $1 \leq j \leq r_{inf}$. It is easy to see that for all $x \in N_{G}(w_i) \cup \{w_i\}$, $1
\in g_i(x)$. So $N_{G}(w_i) \cup \{w_i\}$ induces a clique in
$I_i$. Also for all $x \in V(G) - \{w_i\}$, $2 \in g_i(x)$. So $V(G) - \{w_i\}$ induces a clique in $I_i$. Therefore we infer that
$E(G) \subseteq E(I_i)$ for each $i$ $1 \leq i \leq r_{inf}$.\\
\noindent{\bf Case 2:} When $j=r_{inf}+1$. For any edge $(u,v) \in E(G)$ we consider the following two cases\\
$\noindent\textbf{Subcase 2.1:}$ $\textbf{u}$ $\in \mathcal{O}(P_{inf})$ or $\textbf{v}$ $\in
\mathcal{O}(P_{inf})$. Without loss of generality let $\textbf{u} \in \mathcal{O}(P_{inf})$. Then note that $g_{r_{inf}+1}(u)=[1,2n]$. Clearly $1 \leq \sigma(l(\textbf{v}))
\leq \sigma(r(\textbf{v})) \leq 2n$. Therefore $g_{r_{inf}+1}(u)$ $\cap$
$g_{r_{inf}+1}(v)$ $\neq$ $\emptyset$.\\
$\noindent\textbf{Subcase 2.2:}$ $\textbf{u}$ $\not \in \mathcal{O}(P_{inf})$ and $\textbf{v}$
$\not \in \mathcal{O}(P_{inf})$. Then either $\sigma(l(\textbf{v})) \leq $
$\sigma(l(\textbf{u}))\leq$ $\sigma(r(\textbf{v}))$ or $\sigma(l(\textbf{u})) \leq $
$\sigma(l(\textbf{v}))\leq$ $\sigma(r(\textbf{u}))$. So $g_{r_{inf}+1}(u)$ $\cap$
$g_{r_{inf}+1}(v)$ $\neq$ $\emptyset$.\\
Therefore $(u,v) \in E(I_{r_{inf}+1})$.
\end{proof}

%The following lemma follows from Claim 1 and Claim 2.

\begin{lemma}\label{gen_constr_missing_edge_lemma}
For any $(x, y) \notin E(G)$, $\exists$ $j$, $1 \leq j \leq
r_{inf}+1$, such that $(x, y) \notin E(I_j)$.
\end{lemma}

\begin{proof}
Suppose $(x, y) \notin E(G)$.

\noindent{\bf Case 1:} Either $ \textbf{x} \in \mathcal{O}(P_{inf})$ or $\textbf{y}
\in \mathcal{O}(P_{inf})$. Without loss of generality we can assume that $ \textbf{x} \in \mathcal{O}(P_{inf})$. Then $x = w_i$,
for some $i$ where $1 \leq i \leq r_{inf}$. In $I_i$ as $y \notin
N_G(w_i)$, we have $g_i(x)= g_i(w_i) = [0, 1]$ and $g_i(y) = [2, 3]$. Therefore $g_i(x) \cap
g_i(y) = \emptyset$. Hence $(x, y) \notin E(I_i)$.

\noindent{\bf Case 2:} $\textbf{x} \not \in \mathcal{O}(P_{inf})$ and $\textbf{y} \not
\in \mathcal{O}(P_{inf})$. As $(x,y) \not \in E(G)$ we have either 
$\sigma(r(\textbf{x})) < \sigma(l(\textbf{y}))$ or $\sigma(r(\textbf{y})) < \sigma(l(\textbf{x}))$. Therefore $(x,y) \not \in
E(I_{r_{inf}+1})$.
\end{proof}

\noindent{By combining the above two lemmas we get $E(G) = E(I_1) \cap E(I_2)
\cap \cdots \cap E(I_{r_{inf}+1})$.}

\subsection{Tightness result} Let $G=\overline{(\frac{n}{2})K_2}$, the complement of the perfect matching on n vertices (We will assume that n is even). According to the circular arc representation of $\overline{(\frac{n}{2})K_2}$ described in the introduction, it is easy to see that $|\mathcal{O}(P_{inf})| =
\frac{n}{2} -1$ in $G$. So $box(G) \leq \frac{n}{2}$ by Theorem
\ref{box_general}. But it is known that $box(G) = \frac{n}{2}$
\cite{Roberts}. So the upper bound for boxicity given in Theorem
\ref{box_general} is tight for $\overline{(\frac{n}{2})K_2}$.

\section{An upper bound based on circular cover number}
\begin{defi}\label{circular cover}
For a family $F$ of arcs, a $\textbf{circular \ cover}$ is defined as a subset of arcs of $F$ that can cover the circle. The $\textbf{circular \ cover \ number}$ $L(F)$ of the family of arcs $F$ is defined as the cardinality of the minimum circular cover. If there exists no circular cover for $F$, then $L(F)=\infty$.
\end{defi}
\begin{defi}
 \textbf{Maximum} \textbf{circular cover number} $L_{max}(G)$ is defined as the maximum value of $L(F)$ obtained over all possible family of arcs $F$ that can represent $G$.
\end{defi}
\noindent{}\textbf{Remark:} It may be noted that in many cases getting a circular arc family $F$ that can represent $G$, with higher value of $L(F)$ is preferable. For example see $\cite{tucker}$, where the upper bound shown for the chromatic number is smaller when $L(F)$ is higher.\\\\  
Like the parameter $r_{inf}$, the circular cover number $L(F)$ is also a parameter that appears often in the circular arc graph literature. In this section we will relate $box(G)$ to the maximum circular cover number $L_{max}(G)$.

\begin{lemma}\label{missing point}
There exists $3$ points $P_0$, $P_1$ and $P_2$ on the circle such that given any vertex $v \in V(G)$, \textbf{v} belongs to at most one of the three sets $\mathcal{O}(P_0)$, $\mathcal{O}(P_1)$ and $\mathcal{O}(P_2)$.
\end{lemma}
\begin{proof}
Let $G$ be a circular arc graph having $L_{max}(G)>4$. Let $F= \{\mathcal{A}_1,\mathcal{A}_2,\cdots,
\mathcal{A}_{L_{max}(G)}\}$ be a minimum circular cover of $G$. Let $\mathcal{A}_1,\mathcal{A}_2,\cdots,
\mathcal{A}_{L_{max}(G)}$ be the clockwise ordering of the arcs in $F$ (Recall that the clockwise ordering is defined in definition \ref{clockwise ordering}). Let $l(\mathcal{A}_1)=P_0$ and $r(\mathcal{A}_2)=P_1$. First we will show that $\mathcal{O}(P_0) \cap \mathcal{O}(P_1) = \emptyset$.\\
If possible let $\mathcal{A}' \in \mathcal{O}(P_0) \cap \mathcal{O}(P_1)$.
Then it is easy to see that \textbf{either} we have $\mathcal{A}_1 \subseteq \mathcal{A}'$
and $\mathcal{A}_2 \subseteq \mathcal{A}'$ \textbf{or} $\mathcal{A}'$ should contain the portion of the circle that is not covered by $\mathcal{A}_1 \cup \mathcal{A}_2$. In the former case $\{\mathcal{A}',\mathcal{A}_3,\cdots,\mathcal{A}_{L_{max}(G)}\}$ will be a
circular cover of cardinality $L_{max}(G)-1$, a contradiction. In the latter case $\mathcal{A}_3,\cdots,\mathcal{A}_{L_{max}(G)}$ can be replaced by
$\mathcal{A}'$ in F. So $\{\mathcal{A}_1,\mathcal{A}_2,\mathcal{A}'\}$
will be a circular cover of cardinality $3 < 4 <L_{max}(G)$. Therefore we can infer that $\mathcal{O}(P_0) \cap \mathcal{O}(P_1) = \emptyset$.\\
Next we will show that any arc passing through $P_1$ cannot be clockwise adjacent to any arc passing through $P_0$. If possible let some arc \textbf{x} passing through $P_1$ be clockwise
adjacent to some arc \textbf{y} passing through $P_0$. Then $\{\mathcal{A}_1,\mathcal{A}_2,\textbf{x},\textbf{y}\}$
will be a circular cover of cardinality $4 <L_{max}(G)$, a contradiction.\\
It follows that there exists some
point $P_2$ on the circle such that $\mathcal{O}(P_0) \cap
\mathcal{O}(P_2) = \emptyset$ and $\mathcal{O}(P_1) \cap
\mathcal{O}(P_2) = \emptyset$. Hence the lemma follows.
\end{proof}
\noindent{}We shall construct one interval graph corresponding to each of these three points $P_0$, $P_1$ and $P_2$. Let $\Pi_i$ be the
clockwise ordering of the left and right endpoints of arcs starting from
$P_i$, for $0 \leq i \leq 2$. Note that $\forall \textbf{u} \not \in \mathcal{O}(P_i), \ \Pi_i(l(\textbf{u})) < \Pi_i(r(\textbf{u}))$ and thus $[\Pi_i(l(\textbf{u})),\Pi_i(r(\textbf{u}))]$ is a valid interval on the real line. To construct $I_i$ we map each vertex $u \in V(G)$ to an interval on the real line by the mapping 
\begin{eqnarray*}
  g_i(u) &=& [\Pi_i(l(\textbf{u})), \Pi_i(r(\textbf{u}))] \ \ \ \ \ \ \ \ \ \ if \ \textbf{u} \not \in \mathcal{O}(P_i) \\
  &=& [1, 2n] \ \ \ \ \ \ \ \ \ \ otherwise.
\end{eqnarray*}

\begin{lemma}\label{supgraph1}
For each interval graph $I_i$, $0 \leq i \leq 2$, $E(G) \subseteq
E(I_i)$.
\end{lemma}
\begin{proof}
Let $(u,v) \in E(G)$. We have to consider only
following two cases\\
$\noindent\textbf{Case 1:}$ When \textbf{u} $\not \in \mathcal{O}(P_i)$ and \textbf{v}
$\not \in \mathcal{O}(P_i)$. So either $\Pi_i(l(\textbf{u}))< \Pi_i(l(\textbf{v}))< \Pi_i(r(\textbf{u}))$  or $\Pi_i(l(\textbf{v}))< \Pi_i(l(\textbf{u}))< \Pi_i(r(\textbf{v}))$ which implies that $g_i(u) \cap g_i(v)\neq
\emptyset$.\\
$\noindent\textbf{Case 2:}$ When either \textbf{u} $\in \mathcal{O}(P_i)$ or \textbf{v} $\in
\mathcal{O}(P_i)$. Without loss of generality let $\textbf{u} \in \mathcal{O}(P_i)$. Then note that $g_i(u)=[1,2n]$. Clearly $1 \leq \Pi_i(l(\textbf{v})) \leq \Pi_i(r(\textbf{v})) \leq 2n$. Therefore $g_i(u) \cap g_i(v)\neq
\emptyset$.
\end{proof}

\begin{lemma}\label{missing_edge1}
For any $(u, v) \notin E(G)$, $\exists$ i \ $0 \leq i \leq
2$, such that $(u, v) \notin E(I_i)$.
\end{lemma}
\begin{proof}
In view of Lemma $\ref{missing point}$, at least one of the 3 points $P_0, \ P_1, \ P_2$ is such that both the arcs \textbf{u} and \textbf{v} does not belong to the overlap set of this point. Let $P_i$ be such that $u,v \not \in \mathcal{O}(P_i)$. Without loss of generality let $\Pi_i(l(\textbf{u})) < \Pi_i(l(\textbf{v}))$. Since $(u,v) \not \in E(G)$ we can immediately infer that $\Pi_i(l(\textbf{u})) < \Pi_i(r(\textbf{u})) < \Pi_i(l(\textbf{v})) < \Pi_i(r(\textbf{v}))$. Since in $I_i$, $g_i(u)= [\Pi_i(l(\textbf{u})), \Pi_i(r(\textbf{u}))]$ and $g_i(v)= [\Pi_i(l(\textbf{v})), \Pi_i(r(\textbf{v}))]$ we have $g_i(u) \cap g_i(v) = \emptyset$ and therefore $(u,v) \notin E(I_i)$.
\end{proof}
\noindent{}Combining Lemma $\ref{supgraph1}$ and $\ref{missing_edge1}$ we have the following Theorem.
\begin{theo}\label{box_circ_greater_4}
For a circular arc graph $G$, $box(G) \leq 3$ when maximum circular cover number $L_{max}(G)>4$.
\end{theo}

\bibliographystyle{plain}

\end{document}